\documentclass[12pt,reqno]{amsart} 
\usepackage[T1]{fontenc}
\usepackage{amsfonts}
\usepackage{amssymb}
\usepackage{mathrsfs}
\usepackage[latin1]{inputenc}
\usepackage{amsmath}
\usepackage{paralist}
\usepackage[english]{babel}
\usepackage{setspace}

\newtheorem{theorem}{Theorem}[section]
\newtheorem{lemma}[theorem]{Lemma}

\newtheorem{remark}[theorem]{Remark}
\newtheorem{prop}[theorem]{Proposition}
\newtheorem{example}[theorem]{Example}
\newtheorem{corollary}[theorem]{Corollary}

\numberwithin{equation}{section}
\newcommand{\R}{{\mathbb R}}
\newcommand{\D}{{\mathbb D}}

\newcommand{\C}{{\mathbb C}}
\newcommand{\N}{{\mathbb N}}

\newcommand{\cL}{{\mathcal L}}

\newcommand{\cS}{{\mathcal S}}
\newcommand{\cU}{{\mathcal U}}

\newcommand{\ve}{\varepsilon}

\newcommand{\be}{\beta}

\newcommand{\si}{\sigma}
\newcommand{\su}{\subseteq}

\newcommand\Ker{\mathop{\rm Ker}}

\newcommand{à}{\`a}

\begin{document}

\title{Generalized Ces\`aro operators in weighted Banach spaces of analytic functions with sup-norms}

\author{Angela\,A. Albanese, Jos\'e Bonet and Werner\,J. Ricker}


\address{ Angela A. Albanese\\
Dipartimento di Matematica ``E.De Giorgi''\\
Universit\`a del Salento- C.P.193\\
I-73100 Lecce, Italy}
\email{angela.albanese@unisalento.it}

\address{Jos\'e Bonet \\
Instituto Universitario de Matem\'{a}tica Pura y Aplicada
IUMPA \\
Edificio IDI5 (8E), Cubo F, Cuarta Planta \\
Universidad Politécnica de Valencia \\
E-46071 Valencia, Spain} \email{jbonet@mat.upv.es}

\address{Werner J.  Ricker \\
Math.-Geogr. Fakultät \\
 Katholische Universität
Eichst\"att-Ingol\-stadt \\
D-85072 Eichst\"att, Germany}
\email{werner.ricker@ku.de}

\begin{abstract}
An investigation is made of the generalized Cesàro operators $C_t$, for $t\in [0,1]$, when they act on the space $H(\D)$ of holomorphic functions on the open unit disc $\D$, on the Banach space $H^\infty$ of bounded analytic functions and on the weighted Banach spaces $H_v^\infty$ and $H_v^0$ with their sup-norms. Of particular interest are  the continuity, compactness, spectrum and point spectrum of $C_t$ as well as their linear dynamics and mean ergodicity. 
\end{abstract}
\maketitle

\markboth{A.\,A. Albanese, J. Bonet and W.\,J. Ricker}%
{\MakeUppercase{Generalized Ces\`aro operators}}

\section{Introduction and Preliminaries}

The (discrete) generalized Cesàro operators $C_t$, for $t\in [0,1]$, were first investigated by Rhaly, \cite{R1}, \cite{R2}. The action of $C_t$  from the sequence space $\omega:=\C^{\N_0}$ into itself, with $\N_0:=\{0,1,2,\ldots\}$, is given by
\begin{equation}\label{Ces-op}
	C_tx:=\left(\frac{t^nx_0+t^{n-1}x_1+\ldots +x_n}{n+1}\right)_{n\in\N_0},\quad x=(x_n)_{n\in\N_0}\in\omega.
	\end{equation}
For $t=0$ and with $\varphi:=(\frac{1}{n+1})_{n\in\N_0}$ note that $C_0$ is the diagonal operator
\begin{equation}\label{Dia-op}
	D_\varphi x:= \left(\frac{x_n}{n+1}\right)_{n\in\N_0}, \quad x=(x_n)_{n\in\N_0}\in\omega,
	\end{equation}
and, for $t=1$, that $C_1$ is the classical Ces\`aro averaging operator
\begin{equation}\label{Ces-1}
	C_1x:=\left(\frac{x_0+x_1+\ldots+x_n}{n+1}\right)_{n\in\N_0},\quad x=(x_n)_{n\in\N_0}\in\omega.
\end{equation}
The behaviour of $C_t$ on various sequence spaces has been investigated by many authors. We refer the reader to \cite{R1}, \cite{R2}, \cite{Rho}, to the recent papers \cite{SEl-S}, \cite{YD}, \cite{YM} and to the introduction of the papers \cite{ABR-N}, \cite{CR4} and the references therein.  The  operator $C_1$ was thoroughly investigated on weighted Banach spaces  in \cite{ABR-R}; see also \cite{CR2}. Certain variants of the Ces\`aro operator $C_1$ are considered in \cite{Blasco}, \cite{GGM}. 

Our aim is to investigate the operators $C_t$, for $t\in [0,1]$, when they are suitably interpreted to act on the space $H(\D)$ of holomorphic functions on the open unit disc $\D:=\{z\in\C\ :\ |z|<1\}$, on the Banach space $H^\infty$ of bounded analytic functions and on the weighted Banach spaces $H_v^\infty$ and $H_v^0$ with their sup-norms. 
The space $H(\D)$ is equipped with the topology $\tau_c$ of uniform convergence on the compact subsets of  $\D$. According to \cite[\S 27.3(3)]{23} the space $H(\D)$ is a Fr\'echet-Montel space. A family of norms generating $\tau_c$ is given, for each $0<r<1$, by
\begin{equation}\label{eq.norme-sup}
	q_r(f):=\sup_{|z|\leq r}|f(z)|,\quad f\in H(\D).
\end{equation}
 	
A \textit{weight} $v$ is a continuous, non-increasing function
$v\colon [0,1)\to (0,\infty)$. We extend $v$ to $\D$ by setting $v(z):=v(|z|)$, for $z\in\D$. Note that $v(z)\leq v(0)$ for all $z\in\D$. Given a weight $v$ on $[0,1)$, we define the corresponding \textit{weighted Banach spaces of analytic functions} on $\D$ by
\[
H_v^\infty:=\{f\in H(\D)\, :\, \|f\|_{\infty,v}:=\sup_{z\in\D}|f(z)|v(z)<\infty\},
\]
and
\[
H^0_v:=\{f\in H(\D)\, :\, \lim_{|z|\to 1^-}|f(z)|v(z)=0\},
\]
both endowed with the norm $\|\cdot\|_{\infty,v}$. Since $\|f\|_{\infty,v}\leq v(0)\|f\|_\infty$ whenever $f\in H^\infty$, it is clear that $H^\infty\su H^\infty_v$ with a continuous inclusion.
If $v(z)=1$ for all $z\in\D$, then $H^\infty_v$  coincides with the space $H^\infty$ of all bounded analytic functions on $\D$ with the sup-norm $\|\cdot\|_\infty$ and $H^0_v$ reduces to $\{0\}$.
Moreover, $H^\infty_v\su H(\D)$ continuously. Indeed, fix $0<r<1$. Then $\frac{1}{v(0)}\leq \frac{1}{v(z)}\leq \frac{1}{v(r)}$ for $|z|\leq r$ and so \eqref{eq.norme-sup} implies that
\[
q_r(f)=\sup_{|z|\leq r}\frac{v(z)|f(z)|}{v(z)}\leq \frac{1}{v(r)}\sup_{|z|\leq r}v(z)|f(z)|\leq \frac{1}{v(r)}\|f\|_{\infty, v},\quad f\in H^\infty_v.
\]
 We refer the reader to \cite{Bonet} for a recent survey of such types of weighted Banach spaces and operators between them. 

Whenever necessary we  will identify  a function $f\in H(\D)$ with its sequence of Taylor coefficients $\hat{f}:=(\hat{f}(n))_{n\in\N_0}$ (i.e., $\hat{f}(n):=\frac{f^{(n)}(0)}{n!}$,  for $n\in\N_0$),  so that $f(z)=\sum_{n=0}^\infty \hat{f}(n)z^n$, for $z\in\D$. The linear map $\Phi\colon H(\D)\to\omega$ is defined by
\[
\Phi(f=\sum_{n=0}^\infty \hat{f}(n)z^n):=\hat{f},\quad f\in H(\D).
\]
It is injective (clearly) and continuous. Indeed, for each $m\in\N_0$,
\[
r_m(x):=\max_{0\leq j\leq m}|x_j|,\quad x=(x_j)_{j\in\N_0}\in \omega,
\]
is a continuous seminorm in $\omega$. Fix $0<r<1$, in which case
\begin{align*}
	r_m(\Phi(f))&=\max_{0\leq j\leq m}|\hat{f}(j)|=\max_{0\leq j\leq m}\left|\frac{1}{2\pi i}\int_{|z|=r}\frac{f(z)}{z^{j+1}}\, dz\right|\leq \max_{0\leq j\leq m}\sup_{|z|=r}\frac{|f(z)|}{|z|^j}\\
	&=\max_{0\leq j\leq m}\frac{1}{r^j}q_r(f)\leq \frac{1}{r^m}q_r(f),
\end{align*}
for each $f\in H(\D)$ because $\frac{1}{r^j}\leq \frac{1}{r^m}$ for all $0\leq j\leq m$. Of course, the increasing sequence of seminorms $\{r_m\ :\ m\in\N_0\}$ generates the topology of $\omega$.

We first provide an integral representation of the generalized Cesàro operators $C_t$ defined on $H(\D)$, for $t\in [0,1)$. So, fix $t\in [0,1)$ and define   $C_t\colon H(\D)\to H(\D)$  by $ C_tf(0):=f(0)$ and
\begin{equation}\label{eq.formula-int}
C_tf(z):=\frac{1}{z}\int_0^z\frac{f(\xi)}{1-t\xi}\,d\xi,\ z\in \D\setminus\{0\},
\end{equation}
for every $f\in H(\D)$. It turns out that $C_t$ is continuous on $H(\D)$; see Proposition \ref{Prop-Cont-H(D)}. Moreover, the discrete Cesàro operator $C_t\colon\omega\to\omega$, when restricted to the subspace $\Phi(H(\D))\su\omega$ is transferred to $H(\D)$ as follows.
For a fixed $f\in H(\D)$ we have $f(\xi)=\sum_{n=0}^\infty a_n \xi^n$, for $\xi\in \D$, with $\hat{f}=(a_n)_{n\in\N_0}$ its sequence of Taylor coefficients. Since $\frac{1}{1-t\xi}=\sum_{n=0}^\infty t^n\xi^n$, for $\xi\in\D$, we can form the Cauchy product of the two series, thereby obtaining
\[
\frac{f(\xi)}{1-t\xi}=\sum_{n=0}^\infty(\sum_{k=0}^{n}t^{n-k}a_k)\xi^n, \quad \xi\in \D.
\]
Then \eqref{eq.formula-int} yields
\[
zC_tf(z)=\int_0^z\sum_{n=0}^\infty(\sum_{k=0}^{n}t^{n-k}a_k)\xi^n\,d\xi=\sum_{n=0}^\infty\left(\frac{t^na_0+t^{n-1}a_1+\ldots +a_n}{n+1}\right)z^{n+1},\ z\in\D.
\]
The interchange of the infinite sum and the integral is permissible by uniform convergence of the series. This shows that $C_tf\in H(\D)$ also has the series representation
\begin{align}\label{eq.rapp-serie}
	&C_tf(z)=\sum_{n=0}^\infty\left(\frac{t^na_0+t^{n-1}a_1+\ldots +a_n}{n+1}\right)z^{n}\nonumber\\
	&=\sum_{n=0}^\infty\left(\frac{t^n\hat{f}(0)+t^{n-1}\hat{f}(1)+\ldots +\hat{f}(n)}{n+1}\right)z^{n}=\sum_{n=0}^\infty(C_t^\omega(\hat{f}))_nz^n,
\end{align} 
where the coefficients of the series are precisely as in \eqref{Ces-op}. For the sake of clarity we will denote the discrete generalized Cesàro operator $C_t\colon \omega\to\omega$ by $C_t^\omega$ and reserve the notation $C_t$ for the operator \eqref{eq.formula-int} acting in $H(\D)$. Note that $C_0^\omega=D_\varphi$ (see \eqref{Dia-op}). Moreover, $C_0$ is given by $C_0f(z)=\frac{1}{z}\int_0^z f(\xi)\,d\xi$ for $z\not=0$ and $C_0f(0)=f(0)$, which is the classical Hardy operator in $H(\D)$.

The main results for $C_t$ when acting in the Fr\'echet space $H(\D)$ occur in Proposition \ref{Prop-Cont-H(D)} (continuity), Proposition \ref{NonCompact} (non-compactness), Proposition \ref{Spectrum-H(D)} (spectra) and Proposition \ref{PowerMean-H(D)} (linear dynamics and mean ergodicity). For the analogous information concerning $C_t$ when acting in the weighted Banach spaces $H^\infty_v$ and $H^0_v$ see Proposition \ref{Cont_HInfty_v} and Corollary \ref{Cont_H0_v} (continuity), Proposition \ref{Compact} (compactness), Proposition \ref{Spectrum}  (spectra)  and Proposition \ref{Dyn-Hv} (linear dynamics and mean ergodicity).

We end this section by recalling a few definitions  and some notation concerning locally convex spaces and operators between them. For further details about functional analysis and operator theory relevant to this paper see, for example, \cite{Ed,Gr,J,23,24,Wa}.

Given locally convex Haudorff spaces $X, Y$ (briefly, lcHs) we denote by $\cL(X,Y)$ the space of all continuous linear operators from $X$ into $Y$. If $X=Y$, then we simply write $\cL(X)$ for $\cL(X,X)$. Equipped with the topology of pointwise convergence on $X$ (i.e., the strong operator topology) the lcHs $\cL(X)$ is denoted by $\cL_s(X)$. Equipped with the topology $\tau_b$ of uniform convergence on the bounded subsets of $X$ the lcHs $\cL(X)$ is denoted by $\cL_b(X)$.

Let $X$ be a lcHs space. 
 The identity operator on $X$ is denoted by $I$.  The \textit{transpose operator} of $T\in \cL(X)$ is denoted by  $T'$; it acts from the topological dual space $X':=\cL(X,\C)$ of $X$ into itself. Denote by $X'_\si$ (resp., by $X'_\beta$) the topological dual $X'$ equipped with the weak* topology $\si(X',X)$ (resp., with the strong topology $\beta(X',X)$); see \cite[\S 21.2]{23} for the definition. It is known that $T'\in \cL(X'_\si)$ and $T'\in \cL(X_\be')$,  \cite[p.134]{24}. The bi-transpose operator $(T')'$ of $T$ is simply denoted by $T''$ and belongs to $\cL((X'_\beta)'_\beta)$.

 A linear map $T\colon X\to Y$, with $X,Y$ lcHs', is called \textit{compact} if there exists a neighbourhood $\cU$ of $0$ in $X$ such that $T(\cU)$ is a relatively compact set in $Y$. It is routine to show that necessarily $T\in \cL(X,Y)$. We recall the following well known result; see \cite[Proposition 17.1.1]{J}, \cite[\S 42.1(1)]{24}.
 
 \begin{lemma}\label{L-Comp} Let $X$ be a lcHs. The compact operators are a 2-sided ideal in $\cL(X)$.
 \end{lemma}

Given a lcHs $X$ and $T\in \cL(X)$, the resolvent set $\rho(T;X)$ of $T$ consists of all $\lambda\in\C$ such that $R(\lambda,T):=(\lambda I-T)^{-1}$ exists in $\cL(X)$. The set $\sigma(T;X):=\C\setminus \rho(T;X)$ is called the \textit{spectrum} of $T$. The \textit{point spectrum}  $\sigma_{pt}(T;X)$ of $T$ consists of all $\lambda\in\C$ (also called an eigenvalue of $T$) such that $(\lambda I-T)$ is not injective. Some authors (eg. \cite{Wa}) prefer the subset $\rho^*(T;X)$ of $\rho(T;X)$ consisting of all $\lambda\in\C$ for which there exists $\delta>0$ such that the open disc $B(\lambda,\delta):=\{z\in\C:\, |z-\lambda|<\delta\}\su \rho(T;X)$ and $\{R(\mu,T):\, \mu\in B(\lambda,\delta)\}$ is an equicontinuous subset of $\cL(X)$. Define $\sigma^*(T;X):=\C\setminus \rho^*(T;X)$, which is a closed set with $\sigma(T;X)\su \sigma^*(T;X)$. For the spectral theory of compact operators in lcHs' we refer to \cite{Ed}, \cite{Gr}, for linear dynamics to \cite{B-M}, \cite{G-P} and for mean ergodic operators to \cite{K}, for example.

\section{Continuity, Compactness and Spectrum of $C_t$}

In this section we establish, for $t\in [0,1)$, the continuity of $C_t\colon H(\D)\to H(\D)$ as well as the continuity of $C_t$ from $H^\infty$ (resp., $H^\infty_v$) into $H^\infty$ (resp., $H^\infty_v$). The same is true for $C_t\colon H^0_v\to H^0_v$ whenever $\lim_{r\to 1^-}v(r)=0$. It is also shown that the bi-transpose $C_t''$ of $C_t\in \cL(H^0_v)$ is the generalized Cesàro operator $C_t\in \cL(H^\infty_v)$, provided that $\lim_{r\to 1^-}v(r)=0$. For such weights $v$ it also turns out that both $C_t\in \cL(H^0_v)$  and $C_t\in \cL(H^\infty_v)$ are compact operators (cf. Proposition \ref{Compact}); their spectrum is identified in Proposition \ref{Spectrum}. Of particular interest are the standard weights $v_\gamma(z):=(1-|z|)^\gamma$, for $\gamma>0$ and $z\in\D$.

\begin{prop}\label{Prop-Cont-H(D)} For every $t\in [0,1)$ the  operator $C_t\colon H(\D)\to H(\D)$ is continuous. Moreover, the set $\{C_t\,:\, t\in[0,1)\}$ is equicontinuous in $\cL(H(\D))$.
\end{prop}

\begin{proof} Fix $f\in H(\D)$. Taking
	into account that $C_tf(0)=f(0)$, for all $t\in [0,1)$ and, for each $r\in (0,1)$, that $\sup_{|z|\leq r}|C_tf(z)|=\sup_{|z|=r}|C_tf(z)|$, the formula \eqref{eq.formula-int} implies, for each $z\in\D\setminus\{0\}$, that
	\begin{align*}
		|C_tf(z)|&=\frac{1}{|z|}\left|\int_0^z\frac{f(\xi)}{1-t\xi}d\xi\right|\leq \frac{1}{|z|}|z|\max_{\xi\in [0,z]}\frac{|f(\xi)|}{|1-t\xi|}\\&\leq \frac{1}{1-|z|}\max_{|\xi|\leq |z|}|f(\xi)|=\frac{1}{1-|z|}\max_{|\xi|=|z|}|f(\xi)|,
	\end{align*}
because
$|1-t\xi|\geq 1-t|\xi|\geq 1-|\xi|\geq 1-|z|$, for all $|\xi|\leq |z|$. It follows from the previous inequality, for each $r\in (0,1)$, that
\[
q_r(C_tf)=\sup_{|z|\leq r}|C_tf(z)|\leq \frac{1}{1-r}\sup_{|\xi|\leq r}|f(\xi)|=\frac{1}{1-r}q_r(f);.
\]
see \eqref{eq.norme-sup}.	This implies the result.
\end{proof}

The following example will prove to be  useful in the sequel.
	
\begin{example}\label{Ex.Main}\rm Consider the constant function $f_1(z):=1$,  for every $z\in\D$, in which case  $C_tf_1(0)=f_1(0)=1$ for every $t\in [0,1]$. For $t=0$, it was noted in Section 1 that $C_0$ is
the Hardy operator. In particular, $C_0f_1(z)=1$, for every $z\in\D$. For $t\in (0,1]$, note that $C_tf_1(0)=1$ and
	\[
	C_tf_1(z)=\frac{1}{z}\int_0^z \frac{d\xi}{1-t\xi}=-\frac{1}{tz}\log (1-tz),\quad z\in \D\setminus\{0\}.
	\]
For $t=1$ this shows, in particular, that $C_1(H^\infty)\not\subset H^\infty$, which is well known. For an investigation of the operator $C_1$ acting in $H^\infty$ we refer to \cite{D-S}.
	
	Concerning  $t\in (0,1)$, recall the Taylor series expansion
	\[
	-\log (1-z)=z\sum_{n=0}^\infty \frac{z^n}{n+1},\quad z\in \D,
	\]
	from which it follows that
	\[
	-\frac{\log(1-tz)}{tz}=\sum_{n=0}^\infty\frac{t^n}{n+1}z^n, \quad z\in \D\setminus\{0\},
	\]
	with the series having radius of convergence $\frac{1}{t}>1$. The claim is that $\|C_tf_1\|_\infty=\sup_{|z|<1}|C_tf_1(z)|=-\frac{\log (1-t)}{t}$. Indeed,  $C_tf_1$ is clearly holomorhic in $B(0,\frac{1}{t}):=\{\xi\in \C\, :\, |\xi|<\frac{1}{t}\}$ hence, continuous in $B(0,\frac{1}{t})$, and satisfies $C_tf_1(1)=-\frac{\log(1-t)}{t}$ with  $\lim_{r\to 1^-}C_tf_1(r)=C_tf_1(1)$. On the other hand, for every $z\in\D\setminus\{0\}$ and $t\in (0,1)$ we have that
\[
|C_tf_1(z)|=\left|-\frac{\log (1-tz)}{tz}\right|\leq \sum_{n=0}^\infty\frac{t^n}{n+1}|z|^n\leq \sum_{n=0}^\infty\frac{t^n}{n+1}=-\frac{\log(1-t)}{t}.
\]
This completes the proof of the claim. Observe that $\|C_tf_1\|_\infty>1$. Indeed, define $\gamma(t)=-\log (1-t)-t$, for $t\in [0,1)$. Then $\gamma(0)=0$, $\lim_{t\to 1^-}\gamma(t)=\infty$ and $\gamma'(t)=\frac{1}{1-t}-1=\frac{t}{1-t}$, for $t\in [0,1)$. Since $\gamma'(t)>0$, for  $t\in (0,1)$, it follows that $\gamma$ is strictly increasing and so $\gamma(t)>0$ for all $t\in (0,1)$. This implies that $\|C_tf_1\|_\infty=-\frac{\log (1-t)}{t}>1$ for every $t\in (0,1)$. On the other hand, for $t\in (0,1)$, the inequality $\sum_{n=0}^\infty t^n/(n+1)<\sum_{n=0}^\infty t^n$ implies that
$-\frac{\log (1-t)}{t}<\frac{1}{1-t}$.
So, we have shown that  $\|C_0f_1\|_\infty=1$ and
\[
1<\|C_tf_1\|_\infty<\frac{1}{1-t},\quad t\in (0,1).
\]
	\end{example}

We now turn to the action of $C_t$ in various Banach spaces. For $t=1$ it was noted above that $C_1$ fails to act in $H^\infty$.

\begin{prop}\label{Cont_H_Infty} For $t\in [0,1)$ the  operator $C_t\colon H^\infty\to H^\infty$ is continuous. Moreover, $\|C_0\|_{H^\infty\to H^\infty}=1$ and $$
	\|C_t\|_{H^\infty\to H^\infty}=-\frac{\log (1-t)}{t},\quad t\in (0,1).
	$$
\end{prop}

\begin{proof} Let $f\in H^\infty$ be fixed. Then
	\[
	|C_0f(z)|=\left|\frac{1}{z}\int_0^zf(\xi)d\xi\right|\leq \max_{|\xi|\leq |z|} |f(\xi)|\leq \|f\|_\infty.
	\]
	This implies that $\|C_0\|_{H^\infty\to H^\infty}\leq 1$. On the other hand, $C_0f_1=f_1$ and so we can conclude that $\|C_0\|_{H^\infty\to H^\infty}=1$.
	
	Now let $t\in (0,1)$. Then, for the parametrization $\xi:=sz$, for $s\in (0,1)$, it follows from $|1-stz|\geq 1-|stz|\geq 1-st$ that
	\begin{align*}
		|C_tf(z)|&=\left|\frac{1}{z}\int_0^z\frac{f(\xi)}{1-t\xi}d\xi\right|=\left|\int_0^1\frac{f(sz)}{1-stz}ds\right|\leq \max_{|\xi|\leq |z|}|f(\xi)|\int_0^1\frac{ds}{1-st|z|}\\
		&\leq \|f\|_\infty \int_0^1\frac{ds}{1-st}=-\frac{\log(1-t)}{t}\|f\|_\infty.
	\end{align*}
So, $C_t\in \cL(H^\infty)$ with $\|C_t\|_{H^\infty\to H^\infty}\leq -\frac{\log(1-t)}{t}$. But, $\|C_tf_1\|_\infty= -\frac{\log(1-t)}{t}$. Accordingly, $\|C_t\|_{H^\infty\to H^\infty}= -\frac{\log(1-t)}{t}$.
\end{proof}

\begin{prop}\label{Cont_HInfty_v} Let $v$ be a weight function on $[0,1)$. For each $t\in [0,1)$ the  operator $C_t\colon H_v^\infty\to H_v^\infty$ is continuous. Moreover, $\|C_0\|_{H^\infty_v\to H_v^\infty}=1$ and
	$$
	1\leq \|C_t\|_{H^\infty_v\to H^\infty_v}\leq -\frac{\log(1-t)}{t},\quad t\in (0,1).
	$$
\end{prop}

\begin{proof}
	Recall that $C_tf(0):=f(0)$ for each $f\in H(\D)$ and $t\in [0,1]$. Fix $t\in (0,1)$. Given $f\in H^\infty_v$ and $z\in \D\setminus\{0\}$, observe that
	\begin{align*}
		v(z)|C_tf(z)|&=\frac{v(z)}{|z|}\left|\int_0^z\frac{f(\xi)}{1-t\xi}d\xi\right|=v(z)\left|\int_0^1\frac{f(sz)}{1-stz}ds\right|\\
		&\leq v(z)\int_0^1\frac{|f(sz)|}{|1-stz|}ds\leq \int_0^1\frac{v(sz)|f(sz)|}{|1-stz|}ds\\
		&\leq \|f\|_{\infty,v}\int_0^1\frac{ds}{|1-stz|}\leq\|f\|_{\infty,v}\int_0^1\frac{ds}{1-st|z|}\\
		&=-\frac{\log (1-t|z|)}{t|z|}\|f\|_{\infty,v},
	\end{align*}
	where we used that $v(sz)=v(s|z|)\geq v(|z|)=v(z)$, for $s\in (0,1)$, as $v$ is non-increasing on $(0,1)$ and that $|1-stz|\geq 1-st|z|$, for $s\in (0,1)$. According to the calculations in Example \ref{Ex.Main} we can conclude that
	\[
\|C_tf\|_{\infty, v}=\sup_{z\in\D}|C_tf(z)|v(z)\leq \|f\|_{\infty,v}\sup_{z\in\D}\left[-\frac{\log (1-t|z|)}{t|z|}\right]=-\frac{\log(1-t)}{t}\|f\|_{\infty,v}.
\]
	This implies that $C_t\in \cL(H^\infty_v)$ and $\|C_t\|_{H^\infty_v\to H^\infty_v}\leq -\frac{\log(1-t)}{t}$.
	
	For $t=0$ observe that
	\[
	|C_0f(z)|\leq \int_0^1|f(sz)|ds\leq \max_{|\xi|\leq |z|}|f(\xi)|=\frac{1}{v(z)}\max_{|\xi|=|z|}|f(\xi)|v(\xi)\leq \frac{1}{v(z)}\|f\|_{\infty,v},
	\]
	as $v(\xi)=v(z)$ whenever $|\xi|=|z|$ with $\xi\in \D$. This shows that $\|C_0\|_{H^\infty_v\to H_v^\infty}\leq 1$. Since $C_0f_1=f_1$, it follows that actually $\|C_0\|_{H^\infty_v\to H_v^\infty}=1$.
	
	It remains to show that $\|C_t\|_{H^\infty_v\to H^\infty_v}\geq 1$ for $t\in (0,1)$. To this end,  fix $t\in (0,1)$ and consider the function $g_0(z):=\frac{1}{1-tz}=\sum_{n=0}^\infty t^nz^n$, for $z\in \D$. Then $\|g_0\|_\infty=\frac{1}{1-t}$ and so  $g_0\in H^\infty\subseteq H^\infty_v$.  Moreover, for every $z\in \D\setminus\{0\}$, it is the case that
	\[
	C_tg_0(z)=\frac{1}{z}\int_0^z\frac{d\xi}{(1-t\xi)^2}=\frac{1}{z}\left[\frac{1}{t(1-t\xi)}\right]_0^z=\frac{1}{tz}\left[\frac{1}{1-tz}-1\right]=\frac{1}{1-tz}=g_0(z).
	\]
	It follows that  $\|g_0\|_{\infty,v}=\|C_tg_0\|_{\infty,v}\leq \|C_t\|_{H^\infty_v\to H^\infty_v}\|g_0\|_{\infty,v}$ which implies that $\|C_t\|_{H^\infty_v\to H^\infty_v}\geq 1$.
\end{proof}

\begin{corollary}\label{Cont_H0_v} Let $v$ be a weight function on $[0,1)$ satisfying $\lim_{r\to 1^-}v(r)=0$. 
	For each  $t\in [0,1)$ the operator $C_t\colon H^0_v\to H^0_v$ is continuous and satisfies $\|C_t\|_{H^0_v\to H^0_v}=\|C_t\|_{H^\infty_v\to H^\infty_v}$.
\end{corollary}

\begin{proof} By Proposition \ref{Cont_HInfty_v} and the fact that $H^0_v$ is a closed subspace of $H^\infty_v$,
	to obtain the result  it suffices to establish that $C_t(H_v^0)\subseteq H^0_v$. To this effect,  observe that $H^\infty\subseteq H^0_v$ and that $H^\infty$ is dense in $H^0_v$, as the space of polynomials is dense in $H^0_v$; see Section 1 of \cite{BR} and also \cite{BBG}.  Proposition \ref{Cont_H_Infty} implies that $C_t(H^\infty)\subseteq H^\infty\subseteq H^0_v$. Since $C_t$ acts continuously on $H^\infty_v$, it follows that
	\[
	C_t(H^0_v)=C_t(\overline{H^\infty})\subseteq \overline{C_t(H^\infty)}\subseteq H^0_v.
	\]
	Moreover, $\lim_{r\to 1^-}v(r)=0$ implies that $H^\infty_v$ is canonically isometric to the bidual of $H^0_v$, \cite[Example 2.1]{BS}, and that the bi-transpose $C_t''\colon H^\infty_v\to H_v^\infty$ of $C_t\colon H^0_v\to H^0_v$ coincides with $C_t\colon H^\infty_v\to H^\infty_v$ (see Lemma \ref{L.Main} below), from which the identity $\|C_t\|_{H^0_v\to H^0_v}=\|C_t\|_{H^\infty_v\to H^\infty_v}$  follows.
\end{proof}

\begin{lemma}\label{L.Main} Let $v$ be a weight function on $[0,1)$ satisfying $\lim_{r\to 1^-}v(r)=0$. For each $t\in [0,1)$, the bi-transpose $C''_t\colon H^\infty_v\to H^\infty_v$ of $C_t\colon H^0_v\to H_v^0$ coincides with $C_t\colon H^\infty_v\to H^\infty_v$.
	\end{lemma}

\begin{proof}
	By Proposition \ref{Cont_H_Infty} and Corollary \ref{Cont_H0_v}, together with the fact that   $H^\infty_v$
	is canonically isometric to the bidual of $H^0_v$, both of the operators $C_t''$ and $C_t$ act continuously on $H^\infty_v$.
	
	To show that the bi-transpose $C_t''\colon H^\infty_v\to H^\infty_v$ of $C_t\colon H^0_v\to H^0_v$ coincides with $C_t\colon H^\infty_v\to H^\infty_v$ we proceed via several steps.
	
	\textit{First step}. Given $f\in H(\D)$, its Taylor polynomials $p_k(z)=\sum_{j=0}^k\hat{f}(j)z^j$, $z\in \D$, for $k\in\N_0$, converge to $f$ uniformly on compact subsets of $\D$. That is, $p_k\to f$ in $(H(\D),\tau_c)$ as $k\to\infty$. Accordingly, the averages of $(p_k)_{k\in\N_0}$, that is, the Cesàro means $f_n(z):=\frac{1}{n+1}\sum_{j=0}^np_j(z)$, for $z\in\D$ and $n\in\N_0$, also converge to $f$ in $(H(\D),\tau_c)$ as $n\to\infty$.
	
	\textit{Second step}.  Lemma 1.1 in \cite{BBG} implies, for every $f\in H^\infty_v$ and $n\in\N_0$, that $\|f_n\|_{\infty,v}\leq \|f\|_{\infty,v}$, where $f_n$ is the $n$-th Cesàro mean of $f$, as defined in the \textit{First step}.  Denote by $U_v$ the closed unit ball of $(H^\infty_v,\|\cdot\|_{\infty,v})$. Then, for any given $f\in U_v$, its sequence of   Cesàro means satisfies $(f_n)_{n\in\N_0}\subseteq U_v$ and $f_n\to f$ in $(H(\D),\tau_c)$ as $n\to\infty$.
	
	\textit{Third step}. With the topology of uniform convergence on the compact subsets of $U_v$ denoted by $\tau_c$, let $X:=\{F\in (H_v^\infty)^*:\ F|_{U_v}\ \mbox{is } \tau_c-\mbox{continuous}\}$ be endowed with the norm $\|F\|:=\sup\{|F(f)|:\ f\in U_v\}$. Then \cite[Theorem 1.1(a)]{BS} ensures that $(X,\|\cdot\|)$ is a Banach space and that the evaluation map $\Psi\colon H^\infty_v\to X'$ defined by  $(\Psi(f))(F):=\langle f,F\rangle$, for $F\in X$ and $f\in H^\infty_v$, is an isometric isomorphism onto $X'$ (where $X'$ is the dual  Banach space of $(X,\|\cdot\|)$). Moreover, by \cite[Theorem 1.1(b) and Example 2.1]{BS} the restriction map $R\colon X\to (H^0_v)'$ given by $F\mapsto F|_{H_v^0}$, is also a surjective isometric isomorphism. Therefore, the spaces $H^\infty_v$ and $(H^0_v)''$ are isometrically isomorphic, that is, $X$ and $(H^0_v)'$ are isometrically isomorphic and hence, also $H^\infty_v$ and $(H^0_v)''$ are isometrically isomorphic.
	
	It is easy to see, since the Banach space $X$ above is the predual of $H^\infty_v$, that the evaluation map $\delta_z\in X$, for every $z\in\D$, where $\delta_z\colon f\mapsto f(z)$, for $f\in H^\infty_v$, satisfies $|\langle f,\delta_z\rangle|\leq \|f\|_{\infty,v}/v(z)$. In particular, the linear span $L$ of the set $\{\delta_z:\ z\in\D\}$ separates the points of $H^\infty_v=X'$ and hence, $L$  is dense in $X$. Therefore, the pointwise convergence topology $\tau_p$ on $H^\infty_v$ is Hausdorff and coarser than the $w^*$-topology $\sigma(H^\infty_v, X)$.
	
	\textit{Fourth step}. The closed unit ball $U_v$ of $H^\infty_v$ is a $\tau_c$-compact set  by Montel's theorem, as it is $\tau_c$-bounded and closed. On the other hand, $U_v$ is also $\sigma(H^\infty_v,X)$-compact by the Alaoglu-Bourbaki theorem. Since $\tau_p|_{U_v}$ is coarser than $\tau_c|_{U_v}$ and Hausdorff, we can conclude that  $\tau_p|_{U_v}=\tau_c|_{U_v}$. In the same way, it follows that $\tau_p|_{U_v}=\sigma(H^\infty_v,X)|_{U_v}$. Accordingly, $\tau_p|_{U_v}=\tau_c|_{U_v}=\sigma(H^\infty_v,X)|_{U_v}$.
	
	We are now  ready to prove that $(C_t)''=C_t$. To show this, it suffices to establish that $(C_t)''f=C_tf$ for every $f\in U_v$.

	So, fix $f\in U_v$. With $(f_n)_{n\in\N_0}$ as in the \textit{First step} it follows from there that $f_n\to f$ in $(H(\D),\tau_c)$ as $n\to\infty$ and, by the \textit{Second step}, that $(f_n)_{n\in\N_0}\subseteq U_v$. This implies that $C_tf_n\to C_tf$ in  $(H(\D),\tau_c)$ as $n\to\infty$. Since $C_t\in \cL(H^\infty_v)$ and  $f\in U_v$, it is clear that $C_tf\in H^\infty_v$.
	On the other hand, by the \textit{Fourth step} the sequence $(f_n)_{n\in\N_0}$ also converges to $f$ in $(H^\infty_v,\sigma(H^\infty_v,X))=(H^\infty_v,\sigma(H^\infty_v, (H^0_v)'))$. Since $(C_t)''\colon ((H^0_v)'',\sigma((H^0_v)'',(H^0_v)'))\to ((H^0_v)'',\sigma((H^0_v)'',(H^0_v)')$ is continuous, \cite[\S 8.6]{J}, that is, $(C_t)''\colon (H^\infty_v,\sigma(H^\infty_v, X))\to (H^\infty_v,\sigma(H^\infty_v, X))$ is continuous, it follows that $(C_t)''f_n\to (C_t)''f$ in $(H^\infty_v,\sigma(H^\infty_v, X))$ as $n\to\infty$.
	 Now, $(f_n)_{n\in\N_0}\subset H^\infty\subseteq H^0_v$, as each $f_n$ is a polynomial, and $(C_t)''f_n=C_tf_n$ for every $n\in\N_0$. Moreover, the sequence $C_tf_n\to (C_t)''f$ in $(H(\D),\tau_p)$ as $n\to\infty$. Thus, $(C_t)''f=C_tf$ as desired.
\end{proof}

\begin{prop}\label{Compact} Let $v$ be a weight function satisfying $\lim_{r\to 1^-}v(r)=0$.   For each $t\in [0,1)$, both  of the operators $C_t\colon H^\infty_v\to H^\infty_v$ and $C_t\to H^0_v\to H^0_v$ are  compact.
\end{prop}

\begin{proof} Fix $t\in [0,1)$. Since $H^0_v$ is a closed subspace of $H_v^\infty$ and $C_t(H^0_v)\subseteq H_v^0$ (cf. Corollary \ref{Cont_H0_v}),
	it suffices to show that $C_t\colon H^\infty_v\to H_v^\infty$ is compact. First we establish the following Claim:
	\begin{itemize}
		\item[(*)] Let the sequence $(f_n)_{n\in\N}\subset H^\infty_v$ satisfy $\|f_n\|_{\infty,v}\leq 1$ for every $n\in\N$ and $f_n\to 0$ in $(H(\D),\tau_{c})$ for $n\to\infty$. Then $C_tf_n\to 0$ in $H^\infty_v$.
	\end{itemize}
To prove the Claim, let $(f_n)_{n\in\N}\subset H^\infty_v$ be a sequence as in (*).
Fix $\varepsilon>0$ and select $\delta\in (0,\beta)$, where $\beta:=\min\{1,\frac{\ve(1-t)}{2},\frac{\ve(1-t)}{2v(0)}\}$. Since $\{\xi\in\C\ \ |\xi|\leq (1-\delta)\}$ is a compact subset of $\D$, there exists $n_0\in\N$ such that
\[
\max_{|\xi|\leq 1-\delta}|f_n(\xi)|<\delta, \quad n\geq n_0.
\]
Recall that $C_tf_n(0)=f_n(0)$ for every $n\in\N$. For $z\in \D\setminus\{0\}$ we have seen previously that
\begin{equation*}
	v(z)|C_tf_n(z)|=v(z)\left|\int_0^1\frac{f_n(sz)}{1-stz}ds\right|\leq v(z)\int_0^{1-\delta}\frac{|f_n(sz)|}{|1-stz|}ds+v(z)\int_{1-\delta}^1\frac{|f_n(sz)|}{|1-stz|}ds.
\end{equation*}
Denote the first (resp., second) summand in the right-side of the previous  inequality by $(A_n)$ (resp., by $(B_n)$).
Using the facts that $|1-stz|\geq 1-st |z|\geq \max\{1-s,1-t,1-|z|\}$, for all $s,t\in [0,1)$ and $z\in \D$, and that $v$ is non-increasing on $[0,1)$ it follows, for every $n\geq n_0$, that $\int_0^{1-\delta}|f_n(sz)|\,ds\leq (1-\delta)\max_{|\xi|\leq (1-\delta)}|f_n(\xi)|$ (as $|sz|\leq (1-\delta)$ for all $s\in [0,1-\delta]$) and hence, that
\[
(A_n)\leq  \frac{v(0)(1-\delta)}{1-t}\max_{|\xi|\leq 1-\delta}|f_n(\xi)|<\frac{\ve}{2}.
\]
On the other hand, for every $n\geq n_0$, we have (as $\|f_n\|_{\infty,v}=\sup_{\xi\in\D}v(\xi)|f_n(\xi)|\leq 1$) that
\[
(B_n)=\int_{1-\delta}^1 \frac{v(z)}{v(sz)}\frac{v(sz)|f_n(sz)|}{|1-stz|}\,ds\leq\int_{1-\delta}^1\frac{\|f_n\|_{\infty,v}}{1-t}\,ds\leq \frac{\delta}{1-t}<\frac{\ve}{2}.
\]
It follows that $\|C_tf_n\|_{\infty,v}<\ve$ for every $n\geq n_0$. That is, $C_tf_n\to 0$ in $H^\infty_v$ for $n\to\infty$ and so (*) is proved.

The compactness of $C_t\in\cL(H^\infty_v)$ can be deduced from (*) as follows. Let $(f_n)_{n\in\N}\subset H^\infty_v$ be any bounded sequence. There is no loss of generality in assuming that $\|f_n\|_{\infty,v}\leq 1$ for all $n\in\N$. To establish the compactness of $C_t\in\cL(H^\infty_v)$  we need to show that $(C_tf_n)_{n\in\N}$ has a convergent subsequence in $H^\infty_v$.

Since $H^\infty_v\su H(\D)$ continuously, the sequence $(f_n)_{n\in\N}$ is also bounded in the Fr\'echet-Montel space $H(\D)$. Hence, there is a subsequence $g_j:=f_{n_j}$, for $j\in\N$, of $(f_n)_{n\in\N}$ and $f\in H(\D)$ such that $g_j\to f$ in $H(\D)$ with respect to $\tau_c$. In particular, $g_j\to f$ pointwise on $\D$. Since $v(z)|g_j(z)|=v(z)|f_{n_j}(z)|\leq 1$ for all $z\in\D$ and  $j\in\N$, letting $j\to \infty$ it follows that $v(z)|f(z)|\leq 1$ for all $z\in\D$, that is, $f\in H^\infty_v$ with $\|f\|_{\infty,v}\leq 1$. Let $h_j:=\frac{1}{2}(g_j-f)$, for $j\in\N$. Then $\|h_j\|_{\infty, v}\leq 1$, for $j\in\N$, and $h_j\to 0$ in $H(\D)$ with respect to $\tau_c$.  Condition (*) implies that $C_th_j\to 0$ in $H^\infty_v$ from which it follows that $C_tf_{n_j}=C_tg_j=C_t(g_j-f)+C_tf=2C_th_j+C_tf\to C_tf $ in $H^\infty_v$, as desired. 
\end{proof}

\begin{prop}\label{Spectrum} Let $v$ be a weight function on $[0,1)$ satisfying $\lim_{r\to 1^-}v(r)=0$. For each $t\in [0,1)$ the  spectra of $C_t\in \cL(H^\infty_v)$ and of $C_t\in \cL(H^0_v)$ are given by  
	\begin{equation}\label{Sp-pt}
		\sigma_{pt}(C_t;H^\infty_v)=\sigma_{pt}(C_t;H^0_v)=\left\{\frac{1}{m+1}\,: n\in\N_0\right\},
	\end{equation}
and 
\begin{equation}\label{Sp}
		\sigma(C_t;H^\infty_v)=\sigma(C_t;H^0_v)=\left\{\frac{1}{m+1}\,: n\in\N_0\right\}\cup\{0\}.
\end{equation}
\end{prop}

\begin{proof} Let $t\in [0,1)$ be fixed.
	By \cite[Lemma 3.6]{CR4} we know that the point spectrum of the  operator $C_t^\omega\in \cL(\omega)$ is given by $\sigma_{pt}(C_t^\omega;\omega)=\{\frac{1}{m+1}\,:\, m\in\N_0\}$ and, for each $m\in\N_0$,  that the corresponding eigenspace $\Ker(\frac{1}{m+1}I-C_t^\omega)$ is 1-dimensional and is generated by an eigenvector $x^{[m]}=(x_n^{[m]})_{n\in\N_0}\in \ell^1$.
	Since $H^0_v\subseteq H^\infty_v\subseteq H(\D)$ with continuous inclusions and $\Phi\colon H(\D)\to  \omega$ (cf. Section 1) is a continuous embedding, this implies that $\sigma_{pt}(C_t;H^0_v)\subseteq \sigma_{pt}(C_t;H^\infty_v)\subseteq \{\frac{1}{m+1}\,:\, m\in\N_0\}$. Indeed, let $f\in H(\D)\setminus\{0\}$ and $\lambda\in \C$ satisfy $C_tf=\lambda f$. Then $\lambda f(z)=\sum_{n=0}^\infty \widehat{(\lambda f)}(n)z^n=\sum_{n=0}^\infty\lambda  \hat{f}(n)z^n$ and, by \eqref{eq.rapp-serie}, we have that $(C_tf)(z)=\sum_{n=0}^\infty (C_t^\omega \hat{f})_nz^n$. It follows that $C_t^\omega \hat{f}=\lambda \hat{f}$ in $\omega$ with $\hat{f}\not=0$ and so $\lambda\in \sigma_{pt}(C_t^\omega;\omega)=\{\frac{1}{m+1}\ :\ m\in\N_0\}$.
	
	To conclude the proof, it remains to show that $\{\frac{1}{m+1}\,:\, m\in\N_0\}\subseteq \sigma_{pt}(C_t;H^0_v)$. To establish this recall, for each $m\in\N_0$, that the eigenvector $x^{[m]}\in \ell^1$ and hence, the function $g_m(z):=\sum_{n=0}^\infty (x^{[m]})_nz^n$ belongs to $H^0_v$ because $0\leq v(z)|g_m(z)|\leq v(z)\|x^{[m]}\|_{\ell^1}$ for $z\in\D$ and $\lim_{r\to 1^-}v(r)=0$. Moreover, according to \eqref{eq.formula-int} and \eqref{eq.rapp-serie} we have, for each $z\in\D$, that
	\begin{align*}
	C_tg_m(z)&=\sum_{n=0}^{\infty}(C_t^\omega x^{[m]})_nz^n=\sum_{n=0}^{\infty}(\frac{1}{m+1} x^{[m]})_nz^n=\frac{1}{m+1}\sum_{n=0}^\infty (x^{[m]})_nz^n\\&=\frac{1}{m+1}g_m(z).
	\end{align*}
	Thus $g_m$ is an eigenvector of $C_t\in \cL(H^0_v)$ corresponding to the eigenvalue $\frac{1}{m+1}$.
	
	The validity of $\sigma(C_t;H^0_v)=\sigma(C_t;H^\infty_v)=\{\frac{1}{m+1}\,:\, m\in\N_0\}\cup\{0\}$ follows from the fact that $C_t$ is a  compact operator  on both spaces.
\end{proof}

We now investigate the norm of $C_t$ on $H^\infty_v$ for the standard weights $v_\gamma(z):=(1-|z|)^\gamma$, for $\gamma>0$ and $z\in\D$, which satisfy $\lim_{r\to 1^-}v_\gamma(r)=0$.

\begin{prop}\label{Norm-Ct} Let $t\in (0,1)$ and $\gamma>0$. 
	\begin{itemize}
		\item[\rm (i)] The operator norm $\|C_t\|_{H^\infty_{v_\gamma}\to H^\infty_{v_\gamma}}=1$, for every $\gamma\geq 1$.
		\item[\rm (ii)] For each $\gamma\in (0,1)$, the inequality $\|C_t\|_{H^\infty_{v_\gamma}\to H^\infty_{v_\gamma}}\leq \min\{-\frac{\log (1-t)}{t},\frac{1}{\gamma}\}$ is valid.
	\end{itemize}
	\end{prop}

\begin{proof} We adapt  the arguments given for the Cesàro operator $C_1$ in the proof of \cite[Theorem 2.3]{ABR-R}.
	
	Let $\gamma>0$ and $t\in (0,1)$ be fixed. For $f\in H^\infty_{v_\gamma}$ with $\|f\|_{\infty,v_\gamma}=1$ we have
	\begin{align*}
		|C_tf(z)|&=\frac{1}{|z|}\left|\int_0^1\frac{f(sz)}{1-stz}ds\right|\leq \int_0^1\frac{|f(sz)|}{1-st|z|}ds\\
		&\leq \int_0^1\frac{|f(sz)|}{1-s|z|}ds\leq \int_0^1\frac{ds}{(1-s|z|)^{\gamma+1}}=\frac{1}{(1-|z|)^\gamma}\frac{1-(1-|z|)^\gamma}{\gamma |z|},
	\end{align*}
as $z\in\D$ implies that $1-st|z|\geq 1-s|z|$, for $s\in (0,1)$. Accordingly,
\[
v_\gamma(z)|C_tf(z)|=(1-|z|)^\gamma |C_tf(z)|\leq \frac{1-(1-|z|)^\gamma}{\gamma |z|},\quad z\not=0,
\]
and hence,
\[
\|C_tf\|_{\infty,v_\gamma}\leq\frac{1}{\gamma} \sup_{z\in\D}\frac{1-(1-|z|)^\gamma}{|z|}.
\]
Define $\phi(s):=\frac{1-(1-s)^\gamma}{s}$ for $s\in (0,1]$ and $\phi(0)=\gamma$, in which case $\phi$ is continuous. So, the previous inequality yields $\|C_tf\|_{\infty,v_\gamma}\leq \frac{M_\gamma}{\gamma}$, for all $\|f\|_{\infty,v_\gamma}\leq 1$, that is, $\|C_t\|_{H^\infty_{v_\gamma}\to H^\infty_{v_\gamma}}\leq \frac{M_\gamma}{\gamma}$, where $M_\gamma:=\sup_{s\in [0,1]}\phi(s)$. Proposition \ref{Cont_HInfty_v} yields that $1\leq \|C_t\|_{H^\infty_{v_\gamma}\to H^\infty_{v_\gamma}}\leq -\frac{\log (1-t)}{t}$ for $t\in (0,1)$. On page 101 of \cite{ABR-R}
it is shown that $\frac{M_\gamma}{\gamma}\leq 1$ whenever $\gamma\geq 1$ and that $M_\gamma\leq 1$ for all $\gamma\in (0,1)$. The proof of both parts (i) and (ii) follows immediately.
	\end{proof}

\begin{remark}\label{R-Norm}\rm  For each $\gamma>0$ let $v_\gamma(z)=(1-|z|)^\gamma$, for  $z\in\D$.  Proposition \ref{Norm-Ct} implies that $\sup_{0\leq t<1}\|C_t\|_{H^\infty_{v_\gamma}\to H^\infty_{v_\gamma}}<\infty$. Moreover, if $\gamma\geq 1$, then $\|C_t^n\|_{H^\infty_{v_\gamma}\to H^\infty_{v_\gamma}}=1$ for every $n\in\N$; see case (i) in the proof of \cite[Theorem 2.3]{ABR-R} together with the fact that $1\in \sigma_{pt}(C_t, H^\infty_{v_\gamma})$ by Proposition \ref{Spectrum}.
	\end{remark}

Let $n\in\N$ be fixed. Consider  the weight $v(z)=(\log \frac{e}{1-|z|})^{-n}$, for  $z\in\D$, which satisfies
 $v(0)=1$ and $\lim_{|z|\to 1^-}v(z)=0$.

The function $f(z):=[\log (1-z)]^n\in H(\D)$ belongs to $H^\infty_v$. Indeed, for each $z\in\D$, we have that
\[
|\log (1-z)|=\left|-\sum_{n=1}^\infty \frac{z^n}{n}\right|\leq \sum_{n=1}^\infty \frac{|z|^n}{n}=-\log (1-|z|)
\]
and hence, that
$|f(z)|=|\log (1-z)|^n\leq (-\log (1-|z|))^n$.
Since $v$ is given by  $v(z)=(1-\log(1-|z|))^{-n}$ and
$\lim_{|z|\to 1^-}\frac{-\log (1-|z|)}{1-\log(1-|z|)}=1$,
it follows that $\|f\|_{\infty,v}<\infty$ and so $f\in H^\infty_v$.
On the other hand,
\[
C_1f(z)=\frac{1}{z}\int_0^z\frac{(\log (1-\xi)^n)}{1-\xi}d\xi=-\frac{1}{(n+1)z}(\log (1-z))^{n+1},\quad z\in\D.
\]
Accordingly, $C_1f\not\in H^\infty_v$ since
\begin{align*}
\lim_{s\to s^-}v(s)|(C_1)f(s)|&=\frac{1}{n+1}\lim_{s\to 1^-}\left|\frac{(\log(1-s))^{n+1}}{s(1-\log(1-s))^n}\right|\\
&=\frac{1}{n+1}\lim_{s\to 1^-}\left|\left(\frac{\log(1-s)}{1-\log(1-s)}\right)^n\frac{\log (1-s)}{s}\right|=\infty.
\end{align*}
This implies that the Cesàro operator $C_1$ is not well-defined on $H^\infty_v$, that is, $C_1(H^\infty_v)\not \subseteq H^\infty_v$. But, by Proposition \ref{Cont_HInfty_v} the generalized Cesàro operator $C_t\in \cL(H^\infty_v)$ for every $t\in [0,1)$. At this point, the following question arises:
\textit{Is $\sup_{t\in [0,1)}\|C_t\|_{H^\infty_v\to H^\infty_v}<\infty$  for this particular $v$?} Our next two results show that the answer is negative for certain weights $v$, which includes $v(z)=\left(\log\frac{e}{1-|z|}\right)^{-n}$ for $z\in\D$. 

\begin{prop}\label{SupFinite} Let $v$ be a weight function on $[0,1)$ such that  $\sup_{t\in [0,1)}\|C_t\|_{H^\infty_v\to H^\infty_v}<\infty$. Then $C_1\in \cL(H^\infty_v)$.
	\end{prop}

\begin{proof} Proposition \ref{Prop-Cont-H(D)} implies that   $\{C_t\,:\, t\in [0,1)\}$ is equicontinuous in $\cL(H(\D))$. The  claim is that $\lim_{t\to 1^-}C_tf(z)=C_1f(z)$, for every $f\in H(\D)$ and $z\in\D$.

To prove this claim fix $f\in H(\D)$ and $z\in\D\setminus\{0\}$. Recall, for $t\in [0,1)$, that
\[
C_tf(z)=\frac{1}{z}\int_0^z\frac{f(\xi)}{1-t\xi}d\xi=\int_0^1\frac{f(sz)}{1-stz}ds
	\]
	and
	\[
	C_1f(z)=\frac{1}{z}\int_0^z\frac{f(\xi)}{1-\xi}d\xi=\int_0^1\frac{f(sz)}{1-sz}ds.
		\]
Moreover, for each $z\in\D\setminus\{0\}$, we have (as $|1-stz|\geq (1-|z|)$) that
	\[
	\left|\frac{f(sz)}{1-st z}\right|\leq \frac{|f(sz)|}{1-|z|}\leq \frac{1}{1-|z|}\max_{|\xi|\leq |z|}|f(\xi)|, \quad s\in [0,1],
	\]
	and that $\lim_{t\to 1^-}\frac{f(sz)}{1-st z}=\frac{f(sz)}{1-s z}$ for every $s\in [0,1]$. So, we can apply the dominated convergence theorem to conclude that $\lim_{t\to 1^-}C_tf(z)=C_1f(z)$ for $z\in \D\setminus\{0\}$.
		For $z=0$ we have $C_tf(0)=f(0)=C_1f(0)$ for each $f\in H(\D)$ and $t\in [0,1)$. So, for each $f\in H(\D)$, we can conclude that $C_tf\to C_1f$ pointwise on $\D$ for $t\to 1^-$. The claim is thereby established.
	
	We now show that $C_tf\to C_1f$ in $H(\D)$ as $t\to 1^-$ for every $f\in H^\infty_v$. The assumption $\sup_{t\in [0,1)}\|C_t\|_{H^\infty_v\to H_v^\infty}<\infty$ implies that there exists $M>0$ satisfying $\|C_t\|_{H^\infty_v\to H^\infty_v}\leq M$ for every $t\in [0,1)$. Therefore,
	\begin{equation}\label{eq.Stima}
	\sup_{z\in\D}|C_tf(z)|v(z)\leq M\|f\|_{\infty,v}, \quad f\in H^\infty_v,\ t\in [0,1).
	\end{equation}

	Fix $f\in H^\infty_v$. Then $\{C_tf\,:\, t\in [0,1)\}$ is a bounded set in $H(\D)$. Indeed, given $r\in (0,1)$ and $t\in [0,1)$ we have (as $v(r)\leq v(z)$ for all $|z|\leq r$) that
	\[
	q_r(C_tf)=\sup_{|z|\leq r}|C_tf(z)|=\max_{|z|=r}|C_tf(z)|\leq \frac{M}{v(r)}\|f\|_{\infty,v}.
	\]
	So, the set $\{C_tf\,:\, t\in [0,1)\}$ is bounded in the Fr\'echet-Montel space $H(\D)$ and hence, it is relatively compact in $H(\D)$. Since $C_tf\to C_1f$ pointwise on $\D$ for $t\to 1^-$, it follows that $C_tf\to C_1f$ with respect to $\tau_{c}$, that is, in the Fr\'echet space $H(\D)$, for $t\to 1^-$. In particular, $C_1f\in H(\D)$.
	
	Since $H^\infty_v\su H(\D)$ and $C_th\to C_1h$ pointwise on $\D$ as $t\to 1^-$, for every $h\in H(\D)$, letting $t\to 1^-$ in \eqref{eq.Stima} it follows that
	\[
	|C_1f(z)|v(z)\leq M\|f\|_{\infty,v},\quad z\in \D,
	\]
	that is, $\|C_1f\|_{\infty,v}\leq M\|f\|_{\infty,v}$. But, $f\in H^\infty_v$ is arbitrary and so $C_1\in \cL(H^\infty_v)$.
\end{proof}

\begin{prop}\label{SupInfty} For each $n\in\N$, let $v(z)=(\log(\frac{e}{1-|z|}))^{-n}$ for $z\in\D$.  Then  $\sup_{t\in [0,1)}\|C_t\|_{H^\infty_v\to H^\infty_v}=\infty$.
\end{prop}

\begin{proof}
	Apply Proposition \ref{SupFinite} and the discussion prior it.
\end{proof}

\section{Linear dynamics and Mean Ergodicity of   $C_t$ }

The aim of this section is to investigate the mean ergodicity and the linear dynamics of the operators $C_t$, for $t\in [0,1)$, acting on $H(\D)$, $H^\infty_v$ and $H^0_v$

An operator $T\in \cL(X)$, with $X$ a lcHs, is called \textit{power bounded} if $\{T^n\ :\ n\in\N_0\}$ is an equicontinuous subset of $\cL(X)$. For a Banach space $X$, this means that $\sup_{n\in\N_0}\|T^n\|_{X\to X}<\infty$. Given $T\in \cL(X)$, the averages
\[
T_{[n]}:=\frac{1}{n}\sum_{m=1}^nT^m,\quad n\in\N,
\]
are usually called the Cesàro means of $T$. The operator $T$ is said to be \textit{mean ergodic} (resp., \textit{uniformly mean ergodic}) if $(T_{[n]})_{n\in\N}$ is a convergent sequence in $\cL_s(X)$ (resp., in $\cL_b(X)$). It is routine to check that $\frac{T^n}{n}=T_{[n]}-\frac{n-1}{n}T_{[n-1]}$, for $n\geq 2$, and hence, $\tau_s$-$\lim_{n\to\infty}\frac{T^n}{n}=0$ whenever $T$ is mean ergodic. Every power bounded operator on a Fr\'echet-Montel space $X$ is necessarily uniformly mean ergodic, \cite[Proposition 2.8]{ABR0}. Concerning the linear dynamics of $T\in \cL(X)$, with $X$ a  lcHs, the operator $T$ is called \textit{supercyclic} if, for some $z\in X$, the projective orbit $\{\lambda T^nz\ :\ \lambda\in\C,\ n\in\N_0\}$ is dense in $X$. Since the closure of the linear span of a projective orbit is separable, if such a supercyclic operator $T\in\cL(X)$ exists, then $X$ is necessarily  separable.

 Observe that the space $H^\infty_v$ is never separable, \cite[Theorem 1.1]{Lu}. Therefore, every operator $T\in \cL(H^\infty_v)$ is clearly not supercyclic. However, the spaces $H(\D)$, \cite[Theorem 27.2.5]{23}, and $H^0_v$, \cite[Theorem 1.1]{Lu}, for every weight $v$ are always separable. Hence, the problem of supercyclicity for non-zero operators  $T\in \cL(H(\D))$ and $T\in \cL(H^0_v)$ arises.

The following result, \cite[Theorem 6.4]{ABR-N}, is stated here for Banach spaces.

\begin{theorem}\label{Th-ABR} Let $X$ be a Banach space and let $T\in \cL(X)$ be a compact operator such that $1\in \sigma(T;X)$ with $\sigma(T;X)\setminus \{1\}\subseteq \overline{B(0,\delta)}$ for some $\delta\in (0,1)$ and satisfying  $\Ker (I-T)\cap {\rm Im} (I-T)=\{0\}$. Then $T$ is power bounded and uniformly mean ergodic.
\end{theorem}

A consequence of the previous theorem is the following result.

\begin{prop}\label{Dyn-Hv} Let $v$ be a weight function on $[0,1)$ satisfying $\lim_{r\to 1^-}v(r)=0$. For each $t\in [0,1)$ both of  the operators $C_t\in \cL(H^\infty_v)$ and  $C_t\in \cL(H^0_v)$ are power bounded, uniformly mean ergodic and fail to be supercyclic.
\end{prop}

\begin{proof} Fix $t\in [0,1)$. It was already noted that $C_t\in \cL(H^\infty_v)$ cannot be supercyclic. The operator $C_t$ is  compact operator on  both  $H^\infty_v$ and on $H^0_v$ (cf. Proposition \ref{Compact}). Therefore, the compact transpose operators $C_t'\in \cL((H^\infty_v)')$ and $C'_t\in \cL((H^0_v)')$ have the same non-zero eigenvalues as $C_t$ (see, e.g., \cite[Theorem 9.10-2(2)]{Ed}). In view of Proposition \ref{Spectrum} it follows that $\sigma_{pt}(C_t';(H^\infty_v)')=\sigma_{pt}(C_t';(H^0_v)')=\{\frac{1}{m+1}\,:\, m\in\N_0\}$. We can apply \cite[Proposition 1.26]{B-M} to conclude that $C_t$ is not supercyclic on  $H^0_v$.
	
	By Proposition \ref{Spectrum} and its proof (as $x^{[0]}=(t^n)_{n\in\N_0}$) we have that $\Ker (I-C_t)={\rm span}\{g_0\}$, with $g_0(z)=\sum_{n=0}^\infty t^nz^n$, for $z\in\D$. On the other hand, ${\rm Im}(I-C_t)$ is a closed subspace of $H^\infty_v$ (resp., of $H^0_v$), as $C_t$ is compact in $H^\infty_v$ (resp., in $H^0_v$)), and  ${\rm Im}(I-C_t)\subseteq \{g\in H^\infty_v\,:\, g(0)=0\}$ (resp., $\subseteq \{g\in H^0_v\,:\, g(0)=0\}$), because $C_tf(0)=f(0)$ for each $f\in H^\infty_v$ (resp., each $f\in H^0_v$). Moreover,  \cite[Theorem 9.10.1]{Ed} implies that ${\rm codim}\,{\rm Im}(I-C_t)={\rm dim}\Ker (I-C_t)=1$. Accordingly, both ${\rm Im}(I-C_t)$ and $\{g\in H^\infty_v\, :\, g(0)=0\}=\Ker (\delta_0)$ are hyperplanes, where $\delta_0\in (H^\infty_v)'$ is the linear evaluation functional $f\mapsto f(0)$, for $f\in H^\infty_v$. It follows that necessarily ${\rm Im}(I-C_t)=\{g\in H^\infty_v\, :\, g(0)=0\}$.
	
	Let $h\in {\rm Im}(I-C_t)\cap\Ker (I-C_t)$. Then $h(0)=0$ and there exists $\lambda\in \C$ such that $h=\lambda g_0$. This yields that $0=h(0)=\lambda g_0(0)=\lambda$. Hence, $h=0$. So, ${\rm Im}(I-C_t)\cap\Ker (I-C_t)=\{0\}$.
	
	 Proposition \ref{Spectrum} implies that $1\in \sigma(C_t; H^\infty_v)=\sigma(C_t; H^0_v)=\{\frac{1}{m+1}\,;\, m\in\N_0\}\cup\{0\}$. Consequently, for $\delta=\frac{1}{2}$, all the assumptions of Theorem \ref{Th-ABR} are satisfied. So, we can conclude that $C_t$ is power bounded and uniformly mean ergodic on  both $H^\infty_v$ and on $H^0_v$.
\end{proof}

In contrast to the compactness of $C_t$ acting in the Banach spaces  $H^\infty_v$ and $H^0_v$ (cf. Proposition \ref{Compact}) the situation for the Fr\'echet space $H(\D)$ is different.

\begin{prop}\label{NonCompact} For each $t\in [0,1)$ the  operator $C_t\colon H(\D)\to H(\D)$ is an isomorphism and, hence, it is not compact.
	\end{prop}

\begin{proof} Fix $t\in [0,1)$.
	Consider the operator $T_t\colon H(\D)\to H(\D)$, for $f\in H(\D)$, given by
	\[
	T_tf(z):= (1-tz)(zf(z))'=(1-tz)(f(z)+zf'(z)),\quad  z\in \D.
	\]
	Then $T_t$ is clearly well-defined. Moreover, its graph is closed. Indeed, for a given sequence $(f_n)_{n\in\N}\subset H(\D)$, suppose that $f_n\to f$ in $H(\D)$ and $T_tf_n\to g$ in $H(\D)$. Since  multiplication operators (by elements from $H(\D)$)
and the differentiation operator are continuous on $H(\D)$ and the evaluation functionals at points of $\D$ belong to $H(\D)'$, it follows that
	  $f'_n\to f'$ in $H(\D)$ and hence, $T_tf_n=(1-tz)(f_n+zf'_n)\to (1-tz)(f+zf')=T_tf$ in $H(\D)$. Accordingly, $g=T_tf$. Since $H(\D)$ is a Fr\'echet space, the closed graph theorem, \cite[Corollary 5.4.3]{J}, implies that $T_t\in \cL(H(\D))$.
	
	Finally, it is routine to verify that $C_t\circ T_t=T_t\circ C_t=I$. So,  the inverse operator  $C_t^{-1}=T_t\in \cL(H(\D))$ exists and hence, $C_t$ is a bi-continuous isomorphism of $H(\D)$ onto itself. In particular, $C_t$ cannot be compact.
\end{proof}

Let $\Lambda:=\{\frac{1}{n+1}\,:\, n\in\N_0\}$ and $\Lambda_0:=\Lambda\cup\{0\}$. We recall from \cite[Lemma 2.7]{ABR6} the following lemma, which is an extension of a result of Rhoades \cite{Rho}.


\begin{lemma}\label{Lemma-1} For every $\mu\in\C\setminus \Lambda_0$ there exist $\delta=\delta_\mu>0$  and constants  $d_\delta, D_\delta>0$  such  that $\overline{B(\mu,\delta)}\cap \Lambda_0=\emptyset$ and
	\begin{equation}\label{eq.stimS}
		\frac{d_\delta}{n^{\alpha(\nu)}}\leq \prod_{k=1}^n\left|1-\frac{1}{k\nu}\right|\leq \frac{D_\delta}{n^{\alpha(\nu)}},\quad  \forall n\in\N,\ \nu\in B(\mu,\delta),
	\end{equation}
	where $\alpha(\nu):={\rm Re}(\frac{1}{\nu})$.
\end{lemma}


\begin{remark}\label{R.Remark3.5}	\rm
As a direct  application of Lemma \ref{Lemma-1} we obtain,  for every $\mu\in \C\setminus \Lambda_0$, that there exist $\delta>0$ and $d_\delta, D_\delta>0$ such that $\overline{B(\mu, \delta)}\cap \Lambda_0=\emptyset$ and,  for every $\nu\in B(\mu,\delta)$ and  $n\in\N_0$,  we have that
\begin{equation}\label{H(D)-stima}
	d_\delta D_\delta^{-1}\left(\frac{n-h}{n+1}\right)^{\alpha(\nu)}\leq \prod_{j=n-h+1}^{n+1}\left|1-\frac{1}{j\nu}\right|\leq D_\delta d_\delta^{-1}\left(\frac{n-h}{n+1}\right)^{\alpha(\nu)},
	\end{equation}
for all $h\in \{1,\ldots, n-1\}$, where $\alpha(\nu)={\rm Re}(\frac{1}{\nu})$.
\end{remark}

For each $k\in\N$ with $k\geq 2$ define $r_k:=(1-\frac{1}{k})$. Define the norms $\|\cdot\|_k$ and $|||\cdot|||_k$ on $H(\D)$ by
\[
\|f\|_k:=\sum_{n=0}^\infty |\hat{f}(n)|r_k^n,\quad f=\sum_{n=0}^\infty \hat{f}(n)z^n,
\]
and
\[
|||f|||_k:=\sup_{n\in\N_0} |\hat{f}(n)|r_k^n\quad f=\sum_{n=0}^\infty \hat{f}(n)z^n.
\]

\begin{lemma}\label{L-seminorme} Each of the sequences $\{\|\cdot\|_k\}_{k\geq 2}$ and $\{|||\cdot|||_k\}_{k\geq 2}$ is a fundamental system of norms  for $(H(\D),\tau_c)$.
\end{lemma}

\begin{proof}
	Given  $r\in (0,1)$ choose any $k\geq 2$ such that  $0<r<(1-\frac{1}{k})$. Then, for every $f\in H(\D)$,  we have
	\[
	q_r(f)=\sup_{|z|=r}\left|\sum_{n=0}^\infty \hat{f}(n)z^n\right|\leq \sum_{n=0}^\infty |\hat{f}(n)|r^n\leq \sum_{n=0}^\infty |\hat{f}(n)|\left(1-\frac{1}{k}\right)^n=\|f\|_k.
	\]
	On the other hand, given $k\geq 2$, let $r_k:=(1-\frac{1}{k})<(1-\frac{1}{k+1}):=r_{k+1}$. By the Cauchy inequalities, for $n\in\N_0$, we have
	\[
	|\hat{f}(n)|\leq \frac{1}{r_{k+1}^n}\max_{|z|=r_{k+1}}|f(z)|= \frac{1}{r_{k+1}^n} q_{r_{k+1}}(f),\quad f\in H(\D),
	\]
	and hence,
	\[
	\|f\|_{r_k}=\sum_{n=0}^\infty |\hat{f}(n)|r_k^n\leq q_{r_{k+1}}(f)\sum_{n=0}^\infty\left(\frac{r_k}{r_{k+1}}\right)^n= cq_{r_{k+1}}(f),\quad f\in H(\D),
	\]
	with  $c=\frac{1}{1-\frac{r_k}{r_{k+1}}}=k^2>0$ as $\frac{r_k}{r_{k+1}}<1$, which is independent of $f$.
	
	So, the systems $\{q_r\}_{r\in (0,1)}$ and $\{\|\cdot\|_k\}_{k\geq 2}$ are equivalent on $H(\D)$.
	
	Observe, for every $k\geq 2$, that
	\[
	|||f|||_k=\sup_{n\in\N_0}|\hat{f}(n)|r^n_k\leq \sum_{n=0}^\infty |\hat{f}(n)|r_k^n=\|f\|_k,\quad f\in H(\D),
	\]
	and that
	\begin{align*}
	\|f\|_k&=\sum_{n=0}^\infty |\hat{f}(n)|r_k^n=\sum_{n=0}^\infty |\hat{f}(n)|\left(\frac{r_k}{r_{k+1}}\right)^nr^n_{k+1}\\
	&\leq \sup_{n\in\N_0}|\hat{f}(n)|r^n_{k+1}\sum_{n=0}^\infty \left(\frac{r_k}{r_{k+1}}\right)^n= k^2|||f|||_{k+1},
	\end{align*}
	for $f\in H(\D)$, where $\sum_{n=0}^\infty\left(\frac{r_k}{r_{k+1}}\right)^n=k^2$.
	Therefore, the systems $\{\|\cdot\|_k\}_{k\geq 2}$ and $\{|||\cdot|||_k\}_{k\geq 2}$ are equivalent.
\end{proof}

\begin{prop}\label{Spectrum-H(D)} For each $t\in [0,1)$ the spectra of the   operator $C_t\in \cL(H(\D))$ are given by
	\begin{equation}\label{Sp-H(D)}
	\sigma_{pt}(C_t;H(\D))=\sigma(C_t;H(\D))=\Lambda
		\end{equation}
	and
	\begin{equation}\label{Sp-H(D)-1}
		\sigma^*(C_t;H(\D))=\Lambda_0.
	\end{equation}
\end{prop}

\begin{proof} Let $t\in [0,1)$ be fixed.
For any weight function $v$ on $[0,1)$ satisfying $\lim_{r\to 1^-}v(r)=0$, we have $H^\infty_v\subseteq H(\D)$ continuously and $\Phi\colon H(\D)\to \omega$ is  a continuous imbedding. Accordingly, $\sigma_{pt}(C_t;H_v^\infty)\subseteq \sigma_{pt}(C_t;H(\D))\subseteq \Lambda$; see the proof of Proposition \ref{Spectrum}. Since $\sigma_{pt}(C_t;H^\infty_v)=\Lambda$ (cf. Proposition \ref{Spectrum}) and $\sigma_{pt}(C_t^\omega;\omega)=\Lambda$ \cite[Theorem 3.7]{ABR-N}, it  follows that $\sigma_{pt}(C_t;H(\D))=\Lambda$. Moreover, in view of Proposition \ref{Spectrum} above and Theorem 3.7 in \cite{ABR-N},  the eigenspace corresponding to each eigenvalue $\frac{1}{n+1}\in \Lambda$ is 1-dimensional.
By  Proposition \ref{NonCompact}, the operator $C_t\colon H(\D)\to H(\D)$ is a bi-continuous isomorphism and so $0\not\in \sigma(C_t; H(\D))$.

The  claim is that $\C\setminus \Lambda_0\subseteq \rho(C_t;H(\D))$.
To establish this claim,  fix $\nu\in \C\setminus\Lambda_0$. Given $g(z)=\sum_{n=0}^\infty c_nz^n\in H(\D)$, consider the identity
\begin{equation}\label{eq.ident}
	(C_t-\nu I)f(z)=g(z),\quad z\in \D,
	\end{equation}
where $f(z)=\sum_{n=0}^\infty a_n z^n\in H(\D)$ is to be determined. It follows from \eqref{eq.rapp-serie} that
 $C_tf(z)=\sum_{n=0}^\infty (\frac{t^na_0+t^{n-1}a_1+\ldots +a_n}{n+1})z^n$ from which the identity $(C_t-\nu I)f(z)=\sum_{n=0}^\infty (\frac{t^na_0+t^{n-1}a_1+\ldots +a_n}{n+1}-\nu a_n)z^n$ is clear. So, \eqref{eq.ident} is satisfied if and only if
\[
\sum_{n=0}^\infty (\frac{t^na_0+t^{n-1}a_1+\ldots +a_n}{n+1}-\nu a_n)z^n=\sum_{n=0}^\infty c_nz^n, \quad z\in \D,
\]
that is, if and only if
\[
\frac{t^na_0+t^{n-1}a_1+\ldots +a_n}{n+1}-\nu a_n=c_n,\quad n\in\N_0.
\]
In view of this we can argue, as in the proof of \cite[Lemma 3.6]{ABR-N}, to show  that if a function $f\in H(\D)$ exists which satisfies the identity \eqref{eq.ident}, then the Taylor coefficients $(a_n)_{n\in\N_0}$ of $f$ must verify the following equalities
\begin{align}\label{eq.coeff}
	a_0&=\frac{c_0}{1-\nu}\nonumber\\
	a_n&=\frac{c_n}{(\frac{1}{n+1}-\nu)}+\sum_{h=1}^n(-1)^h\frac{\nu^{h-1}t^hc_{n-h}}{(n+1)\prod_{j=n-h+1}^{n+1}(\frac{1}{j}-\nu)}\\
	&=: A_n+B_n, \quad n\geq 1.\nonumber
\end{align}
Observe, for each $n\geq 1$ and $h\in\{1,\ldots,n\}$, that
\[
(-1)^h\prod_{j=n-h+1}^{n+1}\left(\frac{1}{j}-\nu\right)=-\prod_{j=n-h+1}^{n+1}\left(\nu -\frac{1}{j}\right)=-\nu^{h+1}\prod_{j=n-h+1}^{n+1}\left(1-\frac{1}{j\nu}\right)
\]
and so 
\[
B_n=-\sum_{h=1}^n\frac{\nu^{h-1}t^hc_{n-h}}{\nu^{h+1}(n+1)\prod_{j=n-h+1}^{n+1}(1-\frac{1}{j\nu})}=-\frac{1}{\nu^2}\sum_{h=1}^n\frac{t^hc_{n-h}}{(n+1)\prod_{j=n-h+1}^{n+1}(1-\frac{1}{j\nu})}.
\]
Accordingly, to verify the claim we need to prove that the power series $\sum_{n=0}^\infty a_nz^n$ is convergent in $\D$, with $(a_n)_{n\in\N_0}$ defined according to \eqref{eq.coeff}. First, observe that the series $g(z)=\sum_{n=0}^\infty c_nz^n$ is convergent in $\D$ and satisfies
\[
\limsup_{n\to\infty}\sqrt[n]{|c_n|}=\limsup_{n\to\infty}\sqrt[n]{\frac{|c_n|}{|\frac{1}{n+1}-\nu|}}=\limsup_{n\to\infty}\sqrt[n]{|A_n|}.
\]
Therefore, the series $\sum_{n=1}^\infty A_nz^n$ has the same radius of convergence as the series $\sum_{n=0}^\infty c_nz^n$ and hence, it converges in $H(\D)$. Accordingly,  $f_1(z):=\sum_{n=1}^\infty A_nz^n$, for $z\in\D$, belongs to $H(\D)$. On the other hand, the series
\begin{align*}
\sum_{n=1}^\infty B_nz^n&=-\frac{1}{\nu^2}\sum_{n=1}^\infty \sum_{h=1}^n\frac{t^h c_{n-h}}{(n+1)\prod_{j=n-h+1}^{n+1}(1-\frac{1}{j\nu})}\\
&=-\frac{1}{\nu^2}\sum_{h=1}^\infty t^hz^h\sum_{n=h}^\infty\frac{c_{n-h}z^{n-h}}{(n+1)\prod_{j=n-h+1}^{n+1}(1-\frac{1}{j\nu})},\quad z\in\D.
\end{align*}
To establish the convergence of the series $\sum_{n=1}^\infty B_nz^n$ in $H(\D)$, fix $z\in\D\setminus\{0\}$ and $r\in (|z|,1)$. Recall,  for every $n\in\N_0$, that the Taylor coefficients of $g$ satisfy (as $\frac{1}{r}>1$)
\[
|c_n|=\left|\frac{g^{(n)}(0)}{n!}\right|=\left|\frac{1}{2\pi i}\int_{|\xi|=r}\frac{g(\xi)}{\xi^{n+1}}\,d\xi\right|\leq \frac{1}{r^{n}}\max_{|\xi|=r}|g(\xi)|\leq \frac{C}{r^{n+1}}
\]
where $C:=\max_{|\xi|=r}|g(\xi)|$.
 Therefore, setting $\alpha:=\alpha(\nu)={\rm Re}(\frac{1}{\nu})$ and $d:=d_\delta$ and $D:=D_\delta$ for a suitable $\delta>0$ (cf. Remark \ref{R.Remark3.5}), we obtain via \eqref{eq.stimS} and \eqref{H(D)-stima} that
\begin{align*}
&	\sum_{h=1}^\infty t^h|z|^h\sum_{n=h}^\infty \frac{|c_{n-h}|\,|z|^{n-h}}{(n+1)\prod_{j=n-h+1}^{n+1}|1-\frac{1}{j\nu}|}\\
&\leq C\sum_{h=1}^\infty t^h|z|^{h-1}\left(\frac{|z|}{r}d^{-1}(h+1)^{-\alpha-1}+\sum_{n=h+1}^\infty \left(\frac{|z|}{r}\right)^{n-h+1}Dd^{-1}\left(\frac{n+1}{n-h}\right)^\alpha\right)\\
	&=Cd^{-1}\sum_{h=1}^\infty t^h(h+1)^{-\alpha-1}|z|^{h}+ CDd^{-1}\sum_{h=1}^\infty t^h|z|^{h-1}\sum_{n=h+1}^\infty \left(\frac{|z|}{r}\right)^{n-h+1}\left(\frac{n+1}{n-h}\right)^\alpha\\
	&\leq Cd^{-1}\left(\sum_{h=1}^\infty t^h(h+1)^{-\alpha-1}|z|^{h}+D\sum_{h=1}^\infty t^h |z|^{h-1}\max\{1, (2+h)^\alpha\}\sum_{n=h+1}^\infty\left(\frac{|z|}{r}\right)^{n-h+1}\right),
\end{align*}
 which is finite after  observing that if $\alpha\leq 0$, then $\left(\frac{n+1}{n-h}\right)^\alpha=\left(\frac{n-h}{n+1}\right)^{-\alpha}\leq 1$ for every $h\in\N$ and every $n\geq h+1$, whereas if $\alpha>0$, then $(\frac{n+1}{n-h})^\alpha=(1+\frac{h+1}{n-h})^\alpha\leq (2+h)^\alpha$. This implies that the series $\sum_{n=1}^\infty B^nz^n$ converges in $H(\D)$. Accordingly, $f_2(z):=\sum_{n=1}^\infty B_nz^n$, for $z\in\D$, belongs to $H(\D)$.

Set $f(z):=\frac{c_0}{1-\nu}+f_1(z)+f_2(z)$, for $z\in\D$. Then $f\in H(\D)$. Moreover, the arguments above imply that $f$ satisfies  \eqref{eq.ident}. The identities \eqref{eq.coeff} imply that  $f$ is the unique solution of \eqref{eq.ident}.
Accordingly, the inverse operator $(C_t-\nu I)^{-1}\colon H(\D)\to H(\D)$ exists. In particular,  $(C_t-\nu I)^{-1}\in \cL(H(\D))$ as it is the inverse of a continuous linear operator on a Fr\'echet space.

Since $\nu\in \C\setminus\Lambda_0$ is arbitrary and $0\in \rho(C_t;H(\D))$, we can conclude that $\sigma(C_t;H(\D))=\Lambda$.

It remains to show that $\sigma^*(C_t;H(\D))=\Lambda_0$. To establish this,  fix $\mu\in  \C\setminus\Lambda_0$ and observe, by Lemma \ref{Lemma-1}, that there exist $\delta>0$ and constants $d_\delta, D_\delta>0$ such that $\overline{B(\mu,\delta)}\cap\Lambda_0=\emptyset$ and the inequalities \eqref{eq.stimS} and \eqref{H(D)-stima} are satisfied.
We will show that $B(\mu,\delta)\subset \rho(C_t;H(\D))$ and that the set $\{(C_t-\nu I)^{-1}\,:\, \nu\in B(\mu,\delta)\}$ is equicontinuous in $\cL(H(\D))$. To see this,  first observe that the function $\nu\in \overline{B(\mu,\delta)} \mapsto {\rm Re}(\frac{1}{\nu})\in\R$ is continuous and hence,  $\alpha_0:=\max_{\nu\in \overline{B(\mu,\delta)} }\{{\rm Re}(\frac{1}{\nu})\}$ exists. For the sake of simplicity of notation set $d:=d_\delta$ and $D:=D_\delta$.

Let $\nu\in B(\mu,r)$, where $r:=\frac{1}{2}d(\Lambda_0, \overline{B(\mu,\delta)})>0$ has the property that $|\nu-\frac{1}{j}|>r$ for all $j\in\N$. It was proved above, for any fixed $g(z)=\sum_{n=0}^\infty c_nz^n\in H(\D)$, that
\[
(C_t-\nu I)^{-1}g(z)=\frac{c_0}{1-\nu}+\sum_{n=1}^\infty\left(\frac{c_n}{\frac{1}{n+1}-\nu}-\frac{1}{\nu^2}\sum_{h=1}^n\frac{(-1)^ht^hc_{n-h}}{(n+1)\prod_{j=n-h+1}^{n+1}(1-\frac{1}{j\nu})}\right)z^n,
\]
for each $z\in\D.$. So, for $k\geq 2$ fixed, consider the norm $\|\cdot\|_k$ in $H(\D)$. Then we have, via \eqref{eq.coeff}, that
\begin{align*}
&\|(C_t-\nu I)^{-1}g\|_k\\
&\leq \frac{|c_0|}{|1-\nu|}+\sum_{n=1}^\infty\left|\frac{c_n}{\frac{1}{n+1}-\nu}- \frac{1}{\nu^2}\sum_{h=1}^n\frac{(-1)^ht^hc_{n-h}}{(n+1)\prod_{j=n-h+1}^{n+1}(1-\frac{1}{j\nu})}\right|\left(1-\frac{1}{k}\right)^n\\
&\leq \left(\frac{1}{r}\sum_{n=0}^\infty|c_n|\left(1-\frac{1}{k}\right)^n\right)+\frac{1}{|\nu|^2}\sum_{n=1}^\infty\sum_{h=1}^n\frac{t^h|c_{n-h}|}{(n+1)\prod_{j=n-h+1}^{n+1}|1-\frac{1}{j\nu}|}\left(1-\frac{1}{k}\right)^n\\
&=\frac{1}{r}\|g\|_k+\frac{1}{|\nu|^2}\sum_{h=1}^\infty t^h\left(1-\frac{1}{k}\right)^h\sum_{n=h}^\infty\frac{|c_{n-h}|}{(n+1)\prod_{j=n-h+1}^{n+1}|1-\frac{1}{j\nu}|}\left(1-\frac{1}{k}\right)^{n-h}.
\end{align*}
Moreover, \eqref{eq.stimS} and \eqref{H(D)-stima} with $\alpha(\nu)={\rm Re}(\frac{1}{\nu})\leq \alpha_0$ imply, for each $h\in\N$, that
\begin{align*}
	&\sum_{n=h}^\infty \frac{|c_{n-h}|}{(n+1)\prod_{j=n-h+1}^{n+1}|1-\frac{1}{j\nu}|}\left(1-\frac{1}{k}\right)^{n-h}=\sum_{l=0}^\infty \frac{|c_{l}|}{(l+h+1)\prod_{j=l+1}^{l+h+1}|1-\frac{1}{j\nu}|}\left(1-\frac{1}{k}\right)^{l}\\
	&= \frac{|c_0|}{(h+1)\prod_{j=1}^{h+1}|1-\frac{1}{j\nu}|}+\sum_{l=1}^\infty \frac{|c_{l}|}{(l+h+1)\prod_{j=l+1}^{l+h+1}|1-\frac{1}{j\nu}|}\left(1-\frac{1}{k}\right)^{l}\\
	&\leq d^{-1}|c_0|(h+1)^{\alpha(\nu) -1}+d^{-1}D \sum_{l=1}^\infty \frac{|c_l|}{l+h+1}\left(\frac{l+h+1}{l}\right)^{\alpha(\nu)}\left(1-\frac{1}{k}\right)^l\\
	&\leq d^{-1}|c_0|(h+1)^{\alpha_0 -1}+d^{-1}D \sum_{l=1}^\infty \frac{|c_l|}{l+h+1}\left(\frac{l+h+1}{l}\right)^{\alpha_0}\left(1-\frac{1}{k}\right)^l\\
	&\leq \max\{d^{-1},d^{-1}D\}(2+h)^{\alpha_0}\sum_{l=0}^\infty |c_l|\left(1-\frac{1}{k}\right)^l=K(2+h)^{\alpha_0}\|g\|_k,
\end{align*}
with $K:=\max\{d^{-1},d^{-1}D\}$, and hence, since $|\nu|>r$ for all $\nu\in B(\mu,\delta)$, that
\begin{align*}
&\frac{1}{|\nu|^2}\sum_{h=1}^\infty t^h\left(1-\frac{1}{k}\right)^h\sum_{n=h}^\infty\frac{|c_{n-h}|}{(n+1)\prod_{j=n-h+1}^{n+1}|1-\frac{1}{j\nu}|}\left(1-\frac{1}{k}\right)^{n-h}\\
&\leq \frac{K}{r^2}\|g\|_k\sum_{h=1}^\infty t^h\left(1-\frac{1}{k}\right)^h(2+h)^{\alpha_0}=K'\|g\|_k,
\end{align*}
with $K'=\frac{K}{r^2}\sum_{h=1}^\infty t^h\left(1-\frac{1}{k}\right)^h(2+h)^{\alpha_0}<\infty$, by the ratio test, for instance.

We have established, for every $\nu\in B(\mu,\delta)$, that
\[
\|(C_t-\nu I)^{-1}g\|_k\leq (\frac{1}{r}+K')\|g\|_k.
\]
Since $g\in H(\D)$ and $k\geq 2$ are arbitrary, this shows that the set $\{(C_t-\nu I)^{-1}\,:\, \nu\in B(\mu,\delta)\}$ is equicontinuous. Hence, $\sigma^*(C_t;H(\D))=\Lambda_0$.\end{proof}

\begin{prop}\label{PowerMean-H(D)} For each $t\in [0,1)$ the  operator $C_t\colon H(\D)\to H(\D)$ is power bounded, uniformly mean ergodic but, it fails to be supercyclic. Moreover, $	(I-C_t)(H(\D))$ is the closed subspace of $H(\D)$ given by
	\begin{equation}\label{Imm}
		(I-C_t)(H(\D))=\{g\in H(\D)\, :\, g(0)=0\}
	\end{equation}
and we have the decomposition
\begin{equation}\label{dec}
	H(\D)=\Ker (I-C_t)\oplus (I-C_t)(H(\D)).
\end{equation}
\end{prop}

\begin{proof}
Fix $t\in [0,1)$.
We first prove that $C_t$ is power bounded. Once this is established, $C_t$ is necessarily uniform mean ergodic because $H(\D)$ is a Fr\'echet- Montel  space (see \cite[Proposition 2.8]{ABR0}).

Given $k\geq 2$ we have, for every $f\in H(\D)$ and with $r_k:=(1-\frac{1}{k})$, that
\begin{align*}
|||C_tf|||_k&=\sup_{n\in\N_0}\left|\frac{1}{n+1}\sum_{j=0}^n t^{n-j}\hat{f}(j)\right|r_k\leq \sup_{n\in\N_0}\frac{1}{n+1}\sum_{j=0}^n |\hat{f}(j)|r_k^n\\
&\leq\sup_{n\in\N_0}\frac{1}{n+1}\sum_{j=0}^n |\hat{f}(j)|r_k^j\leq \sup_{j\in\N_0}|\hat{f}(j)|r_k^j =|||f|||_k,
\end{align*}
because $r_k^n\leq r_k^j$ for all $j\in\{0,1,\ldots,n\}$.
It follows, for every $n\in\N$, that 
\[
|||C_t^nf|||_k\leq |||f|||_k,\quad f\in H(\D).
\]
Since $k\geq 2$ is arbitrary, the operator  $C_t\in \cL(H(\D))$ is indeed power bounded.

	To establish that  $C_t\colon H(\D)\to H(\D)$ is not supercyclic, note that the continuous embedding  $\Phi\colon H(\D)\to \omega$  has dense range. The operator $C_t^\omega\in \cL(\omega)$ satisfies  $\Phi\circ C_t=C_t^\omega\circ \Phi$ as an identity in $\cL(H(\D),\omega)$, which implies   if $C_t\colon H(\D)\to H(\D)$ is supercyclic, then also $C_t^\omega\colon \omega \to \omega $ must be supercyclic as $\Phi\circ C_t^n=\Phi\circ C_t\circ C_t^{n-1}=C_t^\omega\circ\Phi\circ C_t^{n-1}=\ldots =(C_t^\omega)^n\circ \Phi$, for all $n\in\N$, and  $\Phi(H(\D))$ is dense in $\omega$. A contradition with \cite[Theorem 6.1]{ABR-N}. 


To establish \eqref{Imm} note that $(I-C_t)(H(\D))\subseteq\{g\in H(\D)\, :\ g(0)=0\}$ because $C_tf(0)=f(0)$ for every $f\in H(\D)$. To show the reverse inclusion, let $g\in H(\D)$ satisfy $g(0)=0$. Then $h(z):=zg'(z)+g(z)$, for $z\in\D$, is holomorphic  and $h(0)=0$. Accordingly, also $z\mapsto \frac{h(z)}{z}$, for $z\in \D\setminus\{0\}$, and taking the value $h'(0)$ at $z=0$ is holomorphic in $\D$. Define $f\in H(\D)$ by
\[
f(z):=\frac{1}{tz-1}\int_0^z(1-t\xi)\frac{h(\xi)}{\xi}\,d\xi,\quad z\in\D,
\]
and note that $f(0)=0$. Direct calculation reveals that
\[
\frac{f(z)}{1-tz}-(zf(z))'=h(z)=(zg(z))',\quad z\in\D,
\]
from which it follows that
\[
\int_0^z\frac{f(\xi)}{1-t\xi}\,d\xi-zf(z)=zg(z),\quad z\in\D.
\]
Since $f(0)=0$, we can conclude that
\[
\frac{1}{z}\int_0^z\frac{f(\xi)}{1-t\xi}\,d\xi-f(z)=g(z),\quad z\in\D,
\]
that is, $(C_t-I)f=g$ and so $g\in (I-C_t)(H(\D))$. Hence, \eqref{Imm} is valid.

To show the validity of \eqref{dec} it suffices to repeat the argument given in the proof of Proposition \ref{Dyn-Hv}.
\end{proof}

\textbf{Acknowledgements.} The research of J. Bonet was partially supported by the project
PID2020-119457GB-100 funded by MCIN/AEI/10.13039/501100011033 and by
``ERFD A way of making Europe'' and by the project GV AICO/2021/170.


\bigskip
\bibliographystyle{plain}

\end{document}